\documentclass{article}
\usepackage{amssymb}
\usepackage{amsfonts}
\usepackage{amsmath}

\newtheorem{theorem}{Theorem}

\newtheorem{definition}[theorem]{Definition}

\newtheorem{lemma}[theorem]{Lemma}

\newtheorem{proposition}[theorem]{Proposition}

\newenvironment{proof}[1][Proof]{\noindent\textbf{#1.} }{\ \rule{0.5em}{0.5em}}
\input{tcilatex}
\begin{document}

\title{Conformal maps between pseudo-Finsler spaces}
\author{Nicoleta Voicu \\
"Transilvania" University, 50, Iuliu Maniu str., Brasov, Romania\\
e-mail: nico.voicu@unitbv.ro}
\date{}
\maketitle

\begin{abstract}
The paper aims to initiate a systematic study of conformal mappings between
Finsler spacetimes and, more generally, between pseudo-Finsler spaces. This
is done by extending several results in pseudo-Riemannian geometry which are
necessary for field-theoretical applications and by proposing a technique
which reduces a series of problems involving pseudo-Finslerian conformal
vector fields to their pseudo-Riemannian counterparts. Also, we point out,
by constructing classes of examples, that conformal groups of flat (locally
Minkowskian) pseudo-Finsler spaces can be much richer than both flat
Finslerian and pseudo-Euclidean conformal groups.
\end{abstract}



\section{Introduction}

In field theory, conformal maps are fundamental for our understanding of
spacetime. Moreover, the existence of a conformal vector field on a manifold
can provide valuable information, which can go up to full classification
results, \cite{Alekseevsky}, \cite{Camargo}, \cite{Caminha}, \cite{Kuhnel
survey},\ on the metric structure.

Among the applications of (pseudo-)Finsler geometry, field-theoretical ones
are the most numerous, e.g., \cite{Asanov}, \cite{Barletta}, \cite%
{Bogoslovsky}, \cite{Gibbons}, \cite{Javaloyes}, \cite{Kouretsis}, \cite%
{Kostelecky}, \cite{Stavrinos}, \cite{Vacaru}. But these applications
typically require metrics to be of Lorentzian signature. And, while on
conformal maps between positive definite Finsler spaces there exists quite a
rich literature, \cite{Aikou}, \cite{Asanjarani}, \cite{Cheng}, \cite%
{Bidabad}, \cite{Hashiguchi}, \cite{Matsumoto}, \cite{Youssef}, \cite{Zhang}%
, in pseudo-Finsler spaces, the situation is completely different. Apart
from a very few papers dedicated to the particular case of isometries, \cite%
{Herrera}, \cite{Torrome-isometries} or to a particular metric, \cite{Pavlov}%
, to the best of our knowledge, even basic questions related to conformal
transformations have not been tackled yet.

\bigskip

Conformal groups of pseudo-Finsler metrics have a much more complicated -
and more interesting - structure than both pseudo-Riemannian and Finslerian
conformal groups. To prove this statement, we present in Section \ref%
{BM_example} some classes of examples of flat (locally Minkowski)
pseudo-Finsler spaces whose conformal symmetries depend on arbitrary
functions. Comparatively, in dimension $n\geq 3,$ conformal symmetries of a
pseudo-Euclidean can only be similarities, inversions and compositions
thereof, \cite{Haantjes}, while the only conformal symmetries of a
non-Euclidean flat Finsler space are similarities, \cite{Matveev}. This
hints at the fact that extending results from either pseudo-Riemannian or
Finsler geometry to pseudo-Finsler spaces can be far from straightforward -
and some of these results might very well fail when passing to
pseudo-Finsler spaces.

\bigskip\ 

In the following sections, we focus on two topics:

1. \textit{The behavior of geodesics under conformal mappings.} Here, we
prove that several results in pseudo-Riemannian geometry (which are
fundamental for general relativity)\ can still be extended to pseudo-Finsler
spaces:

- In dimension greater than 1, any mapping between two pseudo-Finsler
structures which is both conformal and projective is a similarity. In other
words, Weyl's statement (e.g., \cite{Cheng}) that \textit{projective and
conformal properties of a metric space univocally determine its metric up to
a dilation factor} remains true in pseudo-Finsler spaces.

- Lightlike geodesics are preserved, up to re-parametrization, under
arbitrary conformal mappings.

- A conservation law for conformal vector fields along lightlike geodesics.

\bigskip

2. \textit{Conformal vector fields}. In positive definite Finsler spaces,
the technique of averaged Riemannian metrics allows one to prove profound
results regarding conformal transformations, by reducing the corresponding
problems to their Riemannian counterparts, \cite{Matveev}. But,
unfortunately, this technique is not available in pseudo-Finsler spaces, as
noticed in \cite{Torrome-isometries}.

Still, dealing with conformal vector fields, we can find a partial
substitute for this method. Given a pseudo-Finslerian metric tensor $g$ on
some manifold $M,$ an associated Riemannian metric is a pseudo-Riemannian
metric $g^{\xi }:=g\circ \xi ,$ where $\xi $ is a vector field on $M.$
Associated Riemannian metrics have a series of appealing properties (e.g.,
smoothness, same signature as $g$) and behave well under conformal
transformations of $g;$ more precisely, we show (Lemma \ref%
{conformal_vector_fields_lemma}) that,\ if $\xi $ is a conformal vector
field for a pseudo-Finsler metric $g,$ then $\xi $ is also a conformal
vector field for $g^{\xi }.$ This way, some results in pseudo-Riemannian
geometry become available in the more general context of Finsler metrics. As
an example, we extend to pseudo-Finsler spaces two results on Killing vector
fields in \cite{Sanchez}.

Also, we prove that any essential conformal vector field of a pseudo-Finsler
metric has to be lightlike at least at a point.

\bigskip

The paper is organized as follows. Section 2 presents some preliminary
notions and results. Section 3 deals with the basic conformality notions and
examples of pseudo-Finslerian conformal maps. Section 4 is devoted to the
behavior of geodesics under conformal transformations. In Sections 5 and 6,
we discuss pseudo-Finslerian conformal vector fields.

\section{Pseudo-Finsler spaces. Finsler spacetimes}

Let $M$ be a $\mathcal{C}^{\mathcal{\infty }}$-smooth, connected manifold of
dimension $n$ and $(TM,\pi ,M),$ its tangent bundle. We denote by $%
(x^{i})_{i=\overline{0,n-1}}$ the coordinates of a point $x\in M$ in a local
chart $(U,\varphi )$ and consider local charts $(\pi ^{-1}(U),\Phi ),$ $\Phi
=(x^{i},y^{i})_{i=\overline{0,n-1}}$ on $TM$ induced by the choice of the
natural basis $\left\{ \partial _{i}\right\} $ in each tangent space. Commas 
$_{,i}$ will denote differentiation with respect to $x^{i}$ and dots $%
_{\cdot i},$ differentiation with respect to $y^{i}.$ The set of sections of
any fibered manifold $E$ over $M$ will be denoted by $\Gamma (E).$

\bigskip

Consider a non-empty open submanifold$\ $ $A\subset TM,$ with $\pi (A)=M$
and $0\not\in A.$ We assume that each $A_{x}:=T_{x}M\cap A,$ $x\in M,$ is a
positive conic set, i.e., $\forall \alpha >0,$ $\forall y\in A_{x}:\alpha
y\in A_{x}.$ Then the triple $(A,\pi _{|A},M),$ where $\pi _{|A}$ is the
restriction of $\pi $ to $A,$ is a fibered manifold over $M.$ For $x\in M$,
elements $y\in A_{x}$ are called \textit{admissible vectors }at $x.$

\begin{definition}
(\cite{Bejancu}): Fix a natural number $0\leq q<n.$ A smooth function $%
L:A\rightarrow \mathbb{R}$ defines a \textit{pseudo-Finsler structure }$%
(M,A,L)$ on $M$ if, at any point $(x,y)\in A$ and in any local chart $(\pi
^{-1}(U),\Phi )$ around $(x,y):$

1)$\ L(x,\alpha y)=\alpha ^{2}L(x,y),$ $\forall \alpha >0$;

2) the matrix\ $g_{ij}(x,y)=\dfrac{1}{2}\dfrac{\partial ^{2}L}{\partial
y^{i}\partial y^{j}}(x,y)$ has $q$ negative and $n-q$ positive eigenvalues.
\end{definition}

The \textit{Finsler Lagrangian (Finslerian energy) }$L$ can always be
prolonged by continuity to the closure $\bar{A}.$ In particular, we can set $%
L(x,0)=0$.

\bigskip

\textbf{Particular cases.}

1)\ If $q=0,$ then the Finsler structure $(M,A,L)$ is called \textit{%
positive definite.} If $A=TM\backslash \{0\},$ then it is called \textit{%
smooth. }A smooth and positive definite pseudo-Finsler structure is a 
\textit{Finsler structure}.

2)\ A pseudo-Finsler space $(M,A,L)$ with $q=n-1,$ is called a \textit{%
Lorentz-Finsler space} or a \textit{Finsler spacetime}.

In a Finsler spacetime, $ds^{2}=L(x,dx)$ is interpreted as spacetime
interval - and it allows the introduction of the basic causality notions.
For any point $x\in M,$ an admissible vector $y\in A_{x}$\ will be called:\ 
\textit{timelike}, if $L(x,y)>0,$ \textit{spacelike}, if $L(x,y)<0$ and 
\textit{null }or \textit{lightlike,} if $L(x,y)=0.$ Accordingly, a curve $%
c:[a,b]\rightarrow M,$ $t\mapsto c(t)$ is called timelike (respectively,
null, spacelike) if its tangent vector $\dot{c}$ is everywhere timelike
(respectively, null, spacelike)\footnote{%
This terminology will be actually used not only in Finsler spacetimes, but
in pseudo-Finsler spaces of arbitrary signature.}.

3)\ A pseudo-Finsler space $(M,A,L)$ is (\textit{pseudo)-Riemannian}, if, in
any local chart, $g_{ij}=g_{ij}(x)$ and \textit{flat (locally Minkowski) }if
around any point of $A,$ there exists a local chart in which $%
g_{ij}=g_{ij}(y)$ only.

A curve on $M$ is called \textit{admissible} if its tangent vector is
everywhere admissible. In the following, we will assume that all the curves
under discussion are admissible. The arc length of a curve $c:t\in \lbrack
a,b]\mapsto (x^{i}(t))$ on $M$ is calculated as $l(c)=\underset{a}{\overset{b%
}{\int }}F(x(t),\dot{x}(t))dt,$ where the \textit{Finslerian norm} $%
F:A\rightarrow \mathbb{R}$ is defined as: $F=\sqrt{\left\vert L\right\vert }.
$

The correspondence $\left( x,y\right) \mapsto g_{(x,y)}$, where%
\begin{equation}
g_{(x,y)}=g_{ij}(x,y)dx^{i}\otimes dx^{j}  \label{metric_tensor}
\end{equation}%
defines a mapping $g:A\rightarrow T_{2}^{0}M,$ (where $T_{2}^{0}M=T^{\ast
}M\otimes T^{\ast }M$), called the \textit{pseudo-Finslerian metric tensor}
attached to $L.$ A pseudo-Finsler metric $g$ can thus be regarded as a
section of the pullback bundle $\pi _{|A}^{\ast }(T_{2}^{0}M).$

In any local chart $(\pi ^{-1}(U),\Phi )$, there hold the equalities:%
\begin{equation}
L_{\cdot i}=2y_{i},~~\ y_{i\cdot j}=g_{ij},  \label{L_i}
\end{equation}%
where $y_{i}=g_{ij}y^{j}.$

On $A^{o}:=\{(x,y)\in A~|~L(x,y)\not=0\},$ it makes sense the \textit{%
angular metric}%
\begin{equation}
h=g-\dfrac{1}{4L}p\otimes p:A^{o}\rightarrow T_{2}^{0}M
\label{angular metric global}
\end{equation}%
where $p:=\dfrac{\partial L}{\partial y^{i}}dx^{i}.$ Using (\ref{L_i}), this
is written locally as:%
\begin{equation}
h=h_{ij}dx^{i}\otimes dx^{j},~\ \ \ \ \ h_{ij}=g_{ij}-\dfrac{y_{i}y_{j}}{L}.
\label{angular metric def}
\end{equation}%
The functions $h_{ij}$ and their contravariant versions $%
h^{ij}=g^{ik}g^{jl}h_{kl}$ obey: 
\begin{equation}
h_{ij}y^{i}=0,~\ \ h^{ij}y_{i}=0.  \label{angular metric prop}
\end{equation}

\bigskip 

Geodesics of $(M,A,L)$ are described (e.g., \cite{Anto}, \cite{Bucataru}),
by the equations:%
\begin{equation}
\dfrac{d^{2}x^{i}}{dt^{2}}+2G^{i}(x,\dot{x})=0,  \label{geodesic_eqn}
\end{equation}%
where the geodesic coefficients%
\begin{equation}
2G^{i}(x,y)=\dfrac{1}{2}g^{ih}(L_{\cdot h,j}y^{j}-L_{,h})
\label{geodesic_spray}
\end{equation}%
are defined for $\left( x,y\right) \in A.$ The \textit{canonical nonlinear
connection\ }$N$ will be understood as a connection on the fibered manifold $%
A$, in the sense of \cite{Sardanashvili}, pp. 30-32, i.e., as a splitting%
\begin{equation*}
TA=HA\oplus VA,
\end{equation*}%
where $VA=\ker d\pi _{|A}$ is called the vertical subbundle and $HA,$ the
horizontal subbundle of the tangent bundle $(TA,\pi _{|A},A)$. The local
adapted basis will be denoted by $(\delta _{i},\dot{\partial}_{i}),$ where $%
\delta _{i}:=\dfrac{\partial }{\partial x^{i}}-G_{~i}^{j}\dfrac{\partial }{%
\partial y^{j}},$ $\dot{\partial}_{i}=\dfrac{\partial }{\partial y^{j}}$ and
its dual basis, by $(dx^{i},\delta y=dy^{i}+G_{~j}^{i}dx^{j}),$ where%
\begin{equation}
G_{~j}^{i}=G_{~\cdot j}^{i}.  \label{nonlinear_conn_coeffs}
\end{equation}%
Every vector field $X\in \mathcal{X}(M)$ can thus be uniquely decomposed as $%
X=hX+vX,$ where $hX:=X^{i}\delta _{i}\in \Gamma (HA)$ and $vX:=Y^{i}\dot{%
\partial}_{i}\in \Gamma (VA).$

By $^{h}:\Gamma (A)\rightarrow \Gamma (HA),$ $v=v^{i}\partial _{i}\mapsto
v^{h}:=v^{i}\delta _{i}$ and $^{v}:\Gamma (A)\rightarrow \Gamma (VA),$ $%
v=v^{i}\partial _{i}\mapsto v^{v}:=v^{i}\dot{\partial}_{i},$ we will mean
the corresponding horizontal and vertical lifts of vector fields.

\bigskip

The \textit{dynamical covariant derivative}, \cite{Bucataru}, p. 34,
determined by the canonical nonlinear connection $N$ becomes, in a
pseudo-Finsler space $(M,A,L)$, a mapping $\nabla :\Gamma (VA)\rightarrow
\Gamma (VA),$ $X\mapsto \nabla X$ on the vertical subbundle $VA;$ it is
given in any local chart by:%
\begin{equation}
\nabla X_{(x,y)}:=(S(X^{j})+G_{~j}^{i}X^{j})_{(x,y)}\dot{\partial}_{i},~\ \
\forall (x,y)\in A,  \label{dynamical covariant derivative}
\end{equation}%
where $X=X^{i}\dot{\partial}_{i}$ and $S:=y^{k}\delta _{k}.$ The operator $%
\nabla $ acts on functions $f:TM\rightarrow \mathbb{R}$ as: $\nabla f=S(f)$,
it is additive and obeys the Leibniz rule with respect to multiplication
with functions.

The complete lift $\xi ^{\mathbf{c}}=\xi ^{i}\partial _{i}+\xi _{,j}^{i}y^{j}%
\dot{\partial}_{i}$ of an admissible vector field $\xi \in \Gamma (A)$ can
be expressed in terms of $\nabla $ as:%
\begin{equation}
\xi ^{\mathbf{c}}=\xi ^{h}+\nabla (\xi ^{v}).  \label{complete_lift_nabla}
\end{equation}%
From the 2-homogeneity in $y$ of the geodesic coefficients $2G^{i},$ it
follows that, along geodesics $c:[a,b]\rightarrow M,$ $t\mapsto (x^{i}(t))$
of $(M,A,L),$ we have, \cite{Bucataru}, p. 108: $\nabla \dot{x}^{i}=0;$
equivalently,%
\begin{equation}
(\nabla \dot{c}^{v})_{(c(t),\dot{c}(t))}=0.  \label{geodesic_eqn_nabla}
\end{equation}

The canonical nonlinear connection $N$ is metrical, that is, for the
vertical lift $g^{v}=g_{ij}\delta y^{i}\otimes \delta y^{j}:\Gamma
(VA)\times \Gamma (VA)\rightarrow \mathbb{R}$ of the metric $g,$ there holds
(\cite{Bucataru}, p. 98), at any $(x,y)\in A:$ 
\begin{equation}
\nabla g^{v}=0,  \label{metricity_nabla}
\end{equation}%
where $(\nabla g^{v})(X,Y)=\nabla (g^{v}(X,Y))-g^{v}(\nabla
X,Y)-g^{v}(X,\nabla Y),$ $\forall X,Y\in \Gamma (VA).$

Another known property which will be used in the following is that $L$ is
constant along horizontal curves, \cite{Crampin}, that is, 
\begin{equation}
X(L)=0,~\ \ \forall X\in \Gamma (HA).  \label{horizontal_derivs_L}
\end{equation}

\section{Basic notions and examples}

\subsection{\label{basic notions}Conformal maps and conformal vector fields}

The notion of conformal map between Finsler spaces is extended in a
straightforward way to pseudo-Finsler spaces; we have to just take care to
the domains of definition of the involved metric tensors.

\begin{definition}
A diffeomorphism $f:M\rightarrow M^{\prime }$ is called a \textit{conformal
map} between two pseudo-Finsler spaces $(M,A,L)$ and $\left( M^{\prime
},A^{\prime },L^{\prime }\right) $ if there exists a function $\sigma
:M\rightarrow \mathbb{R}$ such that:%
\begin{equation}
L^{\prime }\circ df_{|A}=e^{\sigma }L.  \label{conformal_def_L}
\end{equation}
\end{definition}

In Finsler spacetimes, conformal maps preserve the light cones $L=0.$ For
positive definite Finsler spaces, transformations (\ref{conformal_def_L})
coincide with angle-preserving transformations, \cite{Anto}.

A conformal map is a \textit{similarity }if $\sigma =const.$ and an \textit{%
isometry }if $\sigma =1.$

\bigskip

Denoting by $\tilde{A}:=A\cap \left( df^{-1}\right) (A^{\prime })$ the set
where (\ref{conformal_def_L}) makes sense, (\ref{conformal_def_L}) reads:%
\begin{equation}
L^{\prime }(f(x),df_{x}(y))=e^{\sigma (x)}L(x,y),~\ \ \forall (x,y)\in 
\tilde{A}.  \label{conformal_explicit}
\end{equation}

\bigskip

\textbf{Convention.}\textit{\ }In the following, we will assume that $\pi (%
\tilde{A})=M$ (in particular, this implies that $\tilde{A}$ is a fibered
manifold over $A$). Under this assumption, there will be no loss of
generality if we consider that $A^{\prime }=\left( df\right) (A);$ in the
contrary case, we will restrict our discussion to the sets $\tilde{A}$ and $%
\left( df\right) (\tilde{A})=A^{\prime }\cap df(A)$ respectively and
re-denote them by $A$ and $A^{\prime }$.We will denote the restriction $%
df_{|A}:A\rightarrow A^{\prime }$ simply by $df.$

\bigskip

With the notation%
\begin{equation}
\tilde{L}:=L^{\prime }\circ df,  \label{L_tilde_notation}
\end{equation}%
and with the above convention, (\ref{conformal_def_L}) becomes:%
\begin{equation}
\tilde{L}(x,y)=e^{\sigma }L(x,y),~~\forall (x,y)\in A;  \label{L_tilde}
\end{equation}%
this is equivalent to:%
\begin{equation}
\tilde{g}(x,y)=e^{\sigma }g(x,y),~~\forall (x,y)\in A.  \label{g_tilde_g}
\end{equation}

\bigskip

Assume that $f:M\rightarrow M^{\prime }$ is given with respect to two
arbitrary local charts on $M$ and $M^{\prime }$ as: $\tilde{x}^{i}=\tilde{x}%
^{i}(x^{j})$; the differential $df:A\rightarrow A^{\prime },\left(
x,y\right) \mapsto (\tilde{x},\tilde{y})$ is locally expressed as: $\tilde{x}%
^{i}=\tilde{x}^{i}(x^{j}),~\tilde{y}^{i}=\dfrac{\partial \tilde{x}^{i}}{%
\partial x^{j}}y^{j},$ therefore, differentiating (\ref{L_tilde_notation})
twice with respect to $y^{i},$ we find:%
\begin{equation}
\tilde{g}_{ij}(x,y)=\dfrac{\partial \tilde{x}^{k}}{\partial x^{i}}\dfrac{%
\partial \tilde{x}^{l}}{\partial x^{j}}g_{kl}^{\prime }(\tilde{x},\tilde{y}%
),~\ \ \forall (x,y)\in A.  \label{g_tilde_local}
\end{equation}%
In coordinate-free writing, this is:%
\begin{equation}
\tilde{g}=T_{2}^{0}f\circ g^{\prime }\circ df,  \label{g_tilde_global}
\end{equation}%
where $T_{2}^{0}f:T_{2}^{0}M^{\prime }\rightarrow T_{2}^{0}M$ is the mapping
naturally induced by $f$ on the respective tensor powers (giving the
multiplication by the Jacobian matrix of $f$ in (\ref{g_tilde_local})); we
will write this also as:%
\begin{equation}
\tilde{g}:=\left( df\right) ^{\ast }g^{\prime }.  \label{pullback_notation}
\end{equation}

\bigskip

On a pseudo-Finsler space $(M,A,L)$, an admissible vector field $\xi \in
\Gamma (A)$ is called \textit{conformal} if its 1-parameter group $\left\{
\varphi _{\varepsilon }\right\} _{\varepsilon \in I}$ consists of conformal
transformations, i.e., for any $\varepsilon \in I:$ 
\begin{equation}
L\circ d\varphi _{\varepsilon }=e^{\sigma _{\varepsilon }}L,
\label{1-parameter group}
\end{equation}%
where $\sigma _{\varepsilon }:M\rightarrow M$ are smooth functions.

Assume that $\xi \in \Gamma (A)$ is a conformal vector field. Since $%
d\varphi _{\varepsilon }$ is generated by the complete lift $\xi ^{\mathbf{c}%
},$ we get, by differentiating (\ref{1-parameter group}) at $\varepsilon =0$:%
\begin{equation}
\mathcal{L}_{\xi ^{\mathbf{c}}}L=\dfrac{d}{d\varepsilon }|_{\varepsilon
=0}(e^{\sigma _{\varepsilon }}L)=\mu L,  \label{Lie_deriv_L}
\end{equation}%
where $\mu :=\dfrac{d\sigma _{\varepsilon }}{d\varepsilon }|_{\varepsilon
=0}.$

In particular, if $\sigma _{\varepsilon }=1$ for all $\varepsilon ,$ i.e., $%
\xi $ is a \textit{Killing vector field} for $L$, then:\ $\mathcal{L}_{\xi ^{%
\mathbf{c}}}L=0.$

\bigskip

\textbf{Examples. }If $L=L(y):T\mathbb{R}^{n}\rightarrow \mathbb{R}$ is
locally Minkowski, then:

1)\ The \textit{radial vector field} $\xi (x)=x^{i}\partial _{i}$ is a
conformal vector field. This can be checked easily, as $\xi ^{\mathbf{c}%
}=x^{i}\partial _{i}+x_{,j}^{i}y^{j}\dot{\partial}_{i}=x^{i}\partial
_{i}+y^{i}\dot{\partial}_{i}$ and, using the homogeneity of degree 2 of $L,$
we obtain: 
\begin{equation*}
\mathcal{L}_{\xi ^{\mathbf{c}}}L=x^{i}L_{,i}+y^{i}L_{\cdot i}=0+2L=2L.
\end{equation*}%
The flow of $\xi $ consists of the \textit{dilations (homotheties)} $\varphi
_{\varepsilon }:\left( x^{i}\right) \mapsto (e^{\varepsilon }x^{i}).$

2) Any \textit{constant vector field} $\xi _{0}$ is a Killing vector field
for $L=L(y).$ This follows from: $\xi _{0}^{\mathbf{c}}=\xi _{0}^{i}\partial
_{i}$ and: 
\begin{equation*}
\mathcal{L}_{\xi _{0}^{\mathbf{c}}}L=\xi _{0}^{i}L_{,i}=0.
\end{equation*}%
The vector field $\xi _{0}$ generates the translations $\left( x^{i}\right)
\mapsto (x^{i}+\varepsilon \xi _{0}^{i})$.

\subsection{\label{BM_example} Conformal maps between locally Minkowski
spaces}

In Euclidean spaces, Liouville's Theorem states that any conformal
transformation relating two domains of $\mathbb{R}^{n},$ $n>2,$ is a
similarity or the composition between a similarity and an inversion; passing
to pseudo-Euclidean spaces, one has to only add to the picture, \cite%
{Haantjes}, compositions of two inversions.

In Finsler spaces, the situation is even more rigid; it was proven in \cite%
{Matveev} that any conformal map between two non-Euclidean locally Minkowski
Finsler spaces is a similarity. Taking all these into account, one could
reasonably expect that conformal groups of pseudo-Finsler spaces could not
be too rich.

Still, as we will show in the following, one can create whole \textit{%
families} of pseudo-Finsler metrics with conformal symmetries which are not
only non-similarities, but they depend on arbitrary functions. This gives an
affirmative answer, in indefinite signature, to an old and famous question
raised by M. Matsumoto, \cite{Matveev}, namely \textit{whether there exist
two locally Minkowski structures which are conformal to each other.}

\bigskip

For $\dim M=4,$ a first example is actually known from \cite{Pavlov}. This
example can be extended to any dimension, as follows.

\textbf{Example 1: }\textit{Conformal symmetries of Berwald-Moor metrics. }%
Consider, on $M=\mathbb{R}^{n},$ $n>1:$ 
\begin{equation*}
A=\left\{ (x^{i},y^{i})_{i=\overline{0,n-1}}~|~y^{0}y^{1}....y^{n-1}\not=0%
\right\} \subset TM\backslash \{0\}
\end{equation*}%
and the $n$-dimensional Berwald-Moor pseudo-Finsler function (\cite%
{Matsumoto}, pp. 155-156) on $A:$%
\begin{equation}
L(y)=\varepsilon \left\vert y^{0}y^{1}....y^{n-1}\right\vert ^{\tfrac{2}{n}},
\label{BM_Lagrangian}
\end{equation}%
where\footnote{%
The sign $\varepsilon $ is meant to ensure the existence of spacelike
vectors. Our $L$ is, up to a sign, the square of the one in \cite{Matsumoto}
(the latter would be, in our notations, $F$).} $\varepsilon
:=sign(y^{0}y^{1}....y^{n-1})$.

\bigskip

For an arbitrary diffeomorphism of the form%
\begin{equation}
f:\mathbb{R}^{n}\rightarrow \mathbb{R}^{n},~\ \ x=\left(
x^{0},x^{1},...,x^{n-1}\right) \mapsto
(f^{0}(x^{0}),f^{1}(x^{1}),...,f^{n-1}(x^{n-1})),
\label{conformal transf BM}
\end{equation}%
the Jacobian determinant $J(x):=\dfrac{df^{0}}{dx^{0}}\dfrac{df^{1}}{dx^{2}}%
....\dfrac{df^{n-1}}{dx^{n-1}}$ is always nonzero, hence, there is no loss
of generality if we assume that $J(x)>0$, $\forall x\in \mathbb{R}^{n}.$ We
find: 
\begin{equation}
\tilde{L}(y)=L(df(y))=J(x)^{\tfrac{2}{n}}L(y),~\ \forall y\in A,
\label{transformed _BM metric}
\end{equation}%
i.e., $\tilde{L}$ is also defined on $A.$ Moreover, $f$ is a conformal map,
with conformal factor:%
\begin{equation}
\sigma (x)=\dfrac{2}{n}\ln J(x).  \label{conformal factor BM}
\end{equation}%
The Finsler function (\ref{transformed _BM metric}) is locally Minkowski;
more precisely, the coordinate transformation on $\pi ^{-1}(U)$ induced by: $%
(x^{i})=f^{-1}(x^{i^{\prime }})$ brings $\tilde{L}$ to the form $\tilde{L}%
(y)=(y^{0^{\prime }}y^{1^{\prime }}...y^{n-1^{\prime }})^{\tfrac{2}{n}}.$
Yet, $\sigma (x)$ is not only non-constant, but it depends on $n$ arbitrary
functions.

\bigskip

Berwald-Moor metrics are not the only such examples. Here is a much more
general class of flat pseudo-Finsler metrics on $\mathbb{R}^{n},$ $n\geq 2,$
which admit nontrivial conformal symmetries.

\textbf{Example 2: }\textit{Weighted product Finsler functions. }Consider $M=%
\mathbb{R}^{k}\times \mathbb{R}^{n-k}$ and a pseudo-Finsler metric function $%
L:A\rightarrow \mathbb{R}$ (with $A\subset A_{1}\times A_{2},$ $A_{1}\subset
T\mathbb{R}^{k},$ $A_{2}\subset T\mathbb{R}^{n-k}$), of the form:%
\begin{equation}
L=L_{1}^{\alpha }L_{2}^{1-\alpha },  \label{weighted direct product}
\end{equation}%
where $L_{1}:A_{1}\rightarrow \mathbb{R}$ and $L_{2}:A_{2}\rightarrow 
\mathbb{R}$ are pseudo-Finsler functions and $\alpha \in (0,1).$

Assume that $f_{1}:\mathbb{R}^{k}\rightarrow \mathbb{R}^{k},$ $\left(
x^{0},...,x^{k-1}\right) \mapsto (\tilde{x}^{0},...,\tilde{x}^{k-1})$ is a
conformal transformation with non-constant factor $\sigma =\sigma (x),$ such
that $\tilde{L}_{1}=L_{1}\circ df$ is locally Minkowski - and let $L_{2}$ be
completely arbitrary. Then, the transformation 
\begin{equation}
f:\mathbb{R}^{n}\rightarrow \mathbb{R}^{n},~f:=(f_{1},id_{\mathbb{R}^{n-k}})
\label{conformal_transf_product}
\end{equation}%
leads to: $\tilde{L}(y):=L(df(y))=e^{\alpha \sigma (x)}L(y),$ $\forall
y=(y^{0},...,y^{k-1},y^{k},..,y^{n-1})\in \mathbb{R}^{n},$ i.e., $f$ is a
conformal symmetry (which is not\textit{\ }a similarity) of $L.$ The
obtained Finsler function $\tilde{L}$ is obviously locally Minkowski - and
it depends on the choice of the function $f_{1}.$

\bigskip

\textbf{Particular cases} for the choice of $L_{1}$ in (\ref{weighted direct
product}) include:

a) \textit{The case} $k=1.$ In this case, $L_{1}=\lambda (y^{0})^{2},$ for
some $\lambda \in \mathbb{R}$ and therefore, \textit{any} diffeomorphism%
\begin{equation*}
f_{1}:\mathbb{R\rightarrow R},x^{0}\mapsto f_{1}(x^{0}),
\end{equation*}%
serves the purpose, since: $L_{1}(df_{1}(y^{0}))=L_{1}(\dot{f}%
_{1}(x^{0})y^{0})=\dot{f}_{1}^{2}(x^{0})L_{1}(y^{0}).$

b) \textit{The }$k$\textit{-dimensional Minkowski metric:}%
\begin{equation*}
L_{1}(y^{0},...,y^{k-1})=\left( y^{0}\right) ^{2}-\left( y^{1}\right)
^{2}-....-(y^{n})^{2};
\end{equation*}%
then, $f_{1}$ can be, e.g., an inversion.

c) \textit{The }$k$\textit{-dimensional Berwald-Moor metric} can also be
chosen as $L_{1}$. In this case, $f_{1}$ can be any mapping of the form (\ref%
{conformal transf BM}).

\section{Behavior of geodesics under conformal maps}

\textbf{1. Projective-and-conformal mappings. }A diffeomorphism $%
f:M\rightarrow M$ between is called a \textit{projective map }if geodesics
of $L$ coincide, up to re-parametrization, with geodesics of $\tilde{L}%
:=L\circ df.$ In a completely similar manner to the positive definite case (%
\cite{Anto}, pp. 110-111), it follows that the mapping $f$ is projective if
and only if there exists a 1-homogeneous scalar function $P:A\rightarrow 
\mathbb{R}$ such that, in any local chart, 
\begin{equation}
2\tilde{G}^{i}\left( x,y\right) =2G^{i}\left( x,y\right) +P\left( x,y\right)
y^{i},~\ \forall \left( x,y\right) \in A.  \label{projective_rel_G}
\end{equation}

Assume that the projective map $f$ is also conformal, with conformal factor $%
e^{\sigma }.$ Then, a direct calculation using (\ref{geodesic_spray}) shows
that:%
\begin{equation*}
2\tilde{G}^{i}=2G^{i}+\dfrac{1}{2}g^{ih}\left( \sigma _{,k}y^{k}L_{\cdot
h}-\sigma _{,h}L\right) ;
\end{equation*}%
using (\ref{L_i}), this is: 
\begin{equation}
2\tilde{G}^{i}=2G^{i}+\sigma _{,k}y^{k}y^{i}-\dfrac{1}{2}g^{ih}\sigma _{,h}L.
\label{rel_spray_coefficients}
\end{equation}

\bigskip

Based on the properties of the angular metric tensor (\ref{angular metric
global}), we can extend to arbitrary signature a result known in positive
definite Finsler spaces from \cite{Cheng}, \cite{Szilasi}:

\begin{theorem}
If a mapping $f:M\rightarrow M,$ relating two pseudo-Finsler structures on a
manifold $M$ with $\dim M\geq 2$ is both conformal and projective, then $f$
is a similarity.
\end{theorem}

\begin{proof}
Denote by $(M,A,L)$ and $(M,A,\tilde{L})$ the two Finsler structures; that
is, $\tilde{L}=L\circ df$. As $f$ is both conformal and projective,
equalities (\ref{rel_spray_coefficients}) and (\ref{projective_rel_G}) are
both satisfied. Therefore, at any $\left( x,y\right) \in A$ and in any local
chart around $(x,y),$%
\begin{equation}
\sigma _{,k}y^{k}y^{i}-\dfrac{1}{2}g^{ik}\sigma _{,k}L=Py^{i}.  \label{aux0}
\end{equation}%
Now, fix an arbitrary $x\in M$ and an arbitrary open region of $A_{x}$ where 
$L\not=0$; on such a region, it makes sense the angular metric tensor (\ref%
{angular metric global}). Contracting (\ref{aux0}) with $h_{ij}$ and using (%
\ref{angular metric prop}), it remains: $h_{ij}g^{ik}\sigma _{,k}L=0.$
Taking into account that $h_{ij}g^{ik}=\delta _{j}^{k}-\dfrac{y^{k}y_{j}}{L},
$ this becomes:%
\begin{equation}
L\sigma _{,j}-\sigma _{,k}y^{k}y_{j}=0.  \label{aux1}
\end{equation}%
Differentiating with respect to $y^{i},$ we find, by (\ref{L_i}):%
\begin{equation*}
2y_{i}\sigma _{,j}-\sigma _{,i}y_{j}-\sigma _{,k}y^{k}g_{ij}=0.
\end{equation*}%
Now, contract both hand sides \ of the above equality with $h^{ij}.$ Using
again (\ref{angular metric prop}), we get rid of the first and of the second
term. Further, noticing that $h^{ij}g_{ij}=n-1,$ we obtain: $\left(
n-1\right) \sigma _{,k}y^{k}=0.$ But, by hypothesis, $n=\dim M\geq 2$,
therefore:%
\begin{equation*}
\sigma _{,h}y^{h}=0,
\end{equation*}%
which, by differentiation with respect to $y^{k},$ gives that: $\sigma
_{,k}(x)=0.$ As the point $x$ was arbitrarily chosen, we obtain $\sigma
(x)=const.$, q.e.d.
\end{proof}

\bigskip

\textbf{Remark.}\ Substituting $\sigma =const.$ into (\ref{aux0}), we obtain 
$P=0.$ That is, if two pseudo-Finsler metrics $L$ and $\tilde{L}$ are both
conformally and projectively related, then, $2\tilde{G}^{i}=2G^{i}$ -
meaning that their parametrized geodesics coincide.

\bigskip

\textbf{2. Conformal changes and null geodesics. }Generally, conformal maps
do not preserve geodesics. Still, for null geodesics, we can extend a
remarkable result from the semi-Riemannian case :

\begin{proposition}
Null geodesics of two conformally related pseudo-Finsler metrics coincide up
to parametrization.
\end{proposition}

\begin{proof}
Denote by $L$ and $\tilde{L}$ the two conformally related pseudo-Finsler
metrics. Taking into account that, along null geodesics, $L=0$ and
substituting into (\ref{rel_spray_coefficients}), we find that, along these
curves, $2\tilde{G}^{i}=2G^{i}+\sigma _{,k}y^{k}y^{i}.$ Setting $P:=\sigma
_{,k}y^{k}$, we get: $2\tilde{G}^{i}=2G^{i}+Py^{i},$ which means that null
geodesics of the two spaces coincide up to re-parametrization.
\end{proof}

\bigskip

Another result in pseudo-Riemannian geometry, \cite{Kuhnel survey}, which
can be extended to pseudo-Finsler spaces is:

\begin{proposition}
\label{conservation law}Let $\xi \in $ $\Gamma (A)$ be a conformal vector
field for a pseudo-Finsler space $(M,A,L).$ Along any lightlike geodesic $%
c:[a,b]\rightarrow M$, $t\mapsto c(t)$ the quantity $g_{(c(t),\dot{c}(t))}(%
\dot{c}(t),\xi (t))$ is conserved.
\end{proposition}

\begin{proof}
Take an arbitrary lightlike geodesic $c$ on $M$ and denote by $C=(c,\dot{c}%
), $ the lift of $c$ to $TM.$ Under the above made assumption that $c$ is
admissible, we can write $C:[a,b]\rightarrow A.$

Denote, for simplicity: $g:=g_{(c(t),\dot{c}(t))},$ $\nabla X:=(\nabla
X)_{(c(t),\dot{c}(t))}$ for $X\in \Gamma (VA)$ and $\nabla f:=\nabla
f_{(c(t),\dot{c}(t))}$ for smooth functions on $M.$ As $c$ is a geodesic, we
have, by (\ref{geodesic_eqn_nabla}), $\dot{C}=\dot{x}^{i}\delta _{i};$ hence,%
\begin{equation*}
\dfrac{df}{dt}=\dot{C}f=\dot{x}^{i}\delta _{i}f=\nabla f,~\ \ \ \ \forall
f:TM\rightarrow \mathbb{R}.
\end{equation*}%
Applying the above equality to:%
\begin{equation}
f:=g(\dot{c},\xi )=g^{v}(\dot{c}^{v},\xi ^{v}),  \label{rel_g_to_gv}
\end{equation}%
we get: $\dfrac{d}{dt}(g(\dot{c},\xi ))=\nabla (g^{v}(\dot{c}^{v},\xi ^{v}))$
and therefore, 
\begin{equation}
\dfrac{d}{dt}(g(\dot{c},\xi ))=\left( \nabla g^{v}\right) (\dot{c}^{v},\xi
^{v})+g^{v}((\nabla \dot{c}^{v},\xi ^{v})+g^{v}(\dot{c}^{v},\nabla \xi ^{v}).
\label{deriv_g_csi_c}
\end{equation}%
The first term in the right hand side is zero by (\ref{metricity_nabla}).
The second one is also zero since $c$ is a geodesic. It remains to evaluate $%
g(\dot{c}^{v},\nabla \xi ^{v}).$

Since $\xi $ is a conformal vector field, it obeys:\ $\mathcal{L}_{\xi ^{%
\mathbf{c}}}L=\mu L$ for some function $\mu .$ But, by hypothesis, $c$ is
lightlike, i.e., $L$ vanishes along $C.$ It follows:%
\begin{equation}
\mathcal{L}_{\xi ^{\mathbf{c}}}L=0~\text{\ on \ }C.  \label{zero_Lie_deriv}
\end{equation}%
Further, using (\ref{complete_lift_nabla}) for the Lie derivative $\mathcal{L%
}_{\xi ^{\mathbf{c}}}L=$ $\xi ^{\mathbf{c}}(L),$ relation (\ref%
{zero_Lie_deriv}) becomes:%
\begin{equation*}
\xi ^{h}(L)+(\nabla \xi ^{v})(L)=0.
\end{equation*}%
The term $\xi ^{h}(L)$ vanishes by (\ref{horizontal_derivs_L}), which leads
to $(\nabla \xi ^{v})(L)=0.$ In coordinates, this is: $(\nabla \xi
^{i})L_{\cdot i}=0.$ Taking into account (\ref{L_i}), we can write it as: $%
2g_{ij}y^{j}\nabla \xi ^{i}=0.$ Along $C,$ this is equivalent to:%
\begin{equation*}
g^{v}(\dot{c}^{v},\nabla \xi ^{v})=0.
\end{equation*}%
Substituting the latter relation into (\ref{deriv_g_csi_c}), we get:\ $%
\dfrac{d}{dt}g(\dot{c},\xi )=0,$ q.e.d.
\end{proof}

\section{Associated Riemannian metrics a useful lemma}

Consider a pseudo-Finsler space $(M,A,L)$, with metric tensor $%
g:A\rightarrow T_{2}^{0}M.$ For any admissible vector field $\xi \in \Gamma
(A),$ the mapping%
\begin{equation}
g^{\xi }:=g\circ \xi :M\rightarrow T_{2}^{0}M  \label{g^v}
\end{equation}%
defines a pseudo-Riemannian metric on $M,$ called an \textit{associated
(pseu\-do-)\-Rie\-man\-ni\-an metric }or, \cite{Stavrinos}, an \textit{%
osculating (pseudo)-Riemannian metric}.

\bigskip

Here are some immediate properties of metrics $g^{\xi },$ $\xi \in \Gamma
(A) $:

1. $g^{\xi }$ is defined and smooth on the whole base manifold $M$ (even if $%
g$ cannot be defined on the entire $TM\backslash \{0\}$).

2. $g^{\xi }$ has the same signature as $g.$

3. If, in particular, $g=g(x)$ is pseudo-Riemannian, then, all the metrics $%
g^{\xi },$ $\xi \in \Gamma (A)$ coincide (up to projection onto $M$) with $%
g, $ i.e., $g^{\xi }\circ \pi =g.$

4. If $g,\tilde{g}:A\rightarrow T_{2}^{0}M$ are conformally related, with
conformal factor $\sigma =\sigma (x),$ then%
\begin{equation}
\tilde{g}^{\xi }=e^{\sigma }g^{\xi },~\ \forall \xi \in \Gamma (A).
\label{conformal_g_csi}
\end{equation}

\bigskip

Let us analyze conformal point transformations associated with (\ref%
{conformal_g_csi}). With this aim, consider an arbitrary pseudo-Finsler
metric tensor $g$ and - for the moment - an arbitrary diffeomorphism $%
f:M\rightarrow M.$ Assume, as above, that $df(A)=A,$ denote $df:=df_{|A}$
and consider%
\begin{eqnarray}
\tilde{\xi} &:&=df\circ \xi \circ f^{-1}:M\rightarrow A,  \label{csi_tilde}
\\
\tilde{g} &:&=(df)^{\ast }g=T_{2}^{0}f\circ g\circ df:A\rightarrow
T_{2}^{0}M,
\end{eqnarray}%
the corresponding deformations of $\xi $ and $g.$ Noticing that the pullback 
$f^{\ast }(g^{\tilde{\xi}})$ of the pseudo-Riemannian metric $g^{\tilde{\xi}}
$ can be written as: $f^{\ast }(g^{\tilde{\xi}})=T_{2}^{0}f\circ g^{\tilde{%
\xi}}\circ f,$ we get:%
\begin{equation*}
f^{\ast }(g^{\tilde{\xi}})=T_{2}^{0}f\circ (g\circ \tilde{\xi})\circ
f=T_{2}^{0}f\circ g\circ df\circ \xi =\tilde{g}\circ \xi ,
\end{equation*}%
i.e., the left hand side of (\ref{conformal_g_csi}) is: 
\begin{equation}
\tilde{g}^{\xi }=f^{\ast }(g^{\tilde{\xi}}).  \label{f_star_g_v}
\end{equation}%
\bigskip In particular, if $f$ is a conformal map, then:%
\begin{equation}
f^{\ast }(g^{\tilde{\xi}})=e^{\sigma }g^{\xi }.  \label{f_star_g_conformal}
\end{equation}

Using (\ref{f_star_g_conformal}), we obtain:

\begin{lemma}
\label{conformal_vector_fields_lemma}If $\xi :M\rightarrow A$ is a conformal
vector field for a pseudo-Finsler metric structure $(M,A,L)$, with
1-parameter group $\left\{ \varphi _{\varepsilon }\right\} ,$ then:

(i) $\xi $ it is a conformal vector field for the pseudo-Riemannian metric $%
g^{\xi }.$

(ii) The conformal factor relating the pseudo-Finsler metrics $g$ and $%
\tilde{g}=(d\varphi _{\varepsilon })^{\ast }g$ is the same as the conformal
factor relating $g^{\xi }$ and $\varphi _{\varepsilon }^{\ast }(g^{\xi }).$
\end{lemma}

\begin{proof}
\textit{(i) }Set:$\ \tilde{g}:=\left( d\varphi _{\varepsilon }\right) ^{\ast
}g,$ $\tilde{\xi}:=d\varphi _{\varepsilon }\circ \xi \circ \varphi
_{\varepsilon }^{-1}.$ By (\ref{f_star_g_conformal}), we have: $\varphi
_{\varepsilon }^{\ast }(g^{\tilde{\xi}})=e^{\sigma }g^{\xi }.$ But, the
vector field $\xi $ is invariant under its own flow, that is, $\tilde{\xi}%
=\xi $. We find: 
\begin{equation}
\varphi _{\varepsilon }^{\ast }(g^{\xi })=e^{\sigma }g^{\xi },
\label{g_tilde_phi_epsilon}
\end{equation}%
that is, $\xi $ is a conformal vector field for $g^{\xi }.$

\textit{(ii)} The statement follows from (\ref{g_tilde_phi_epsilon}).
\end{proof}

\section{Conformal and Killing vector fields}

Here is another property in pseudo-Riemannian geometry, \cite{Kuhnel survey}%
, which can be extended to pseudo-Finsler spaces:

\begin{proposition}
If a conformal vector field $\xi :M\rightarrow A$ for a pseudo-Finsler space 
$(M,A,L)$ is nowhere lightlike, then, $\xi $ is a Killing vector field for a
conformally related pseudo-Finsler structure.
\end{proposition}

\begin{proof}
As $\xi $ is a conformal vector field for $L,$ we have, at any $(x,y)\in A:$%
\begin{equation}
(\mathcal{L}_{\xi ^{\mathbf{c}}}L)\left( x,y\right) =\mu L(x,y).
\label{csi_y}
\end{equation}%
Using the hypothesis that $L$ is nowhere lightlike, the quantity $\alpha
(x):=L(x,\xi (x))$ does not vanish. Set:%
\begin{equation*}
\tilde{L}(x,y):=\dfrac{1}{\alpha (x)}L(x,y):A\rightarrow \mathbb{R}.
\end{equation*}%
Taking the Lie derivative of $\tilde{L}$: $(\mathcal{L}_{\xi ^{\mathbf{c}}}%
\tilde{L})=\mathcal{L}_{\xi ^{\mathbf{c}}}(\dfrac{1}{\alpha })L+\dfrac{1}{%
\alpha }\mathcal{L}_{\xi ^{\mathbf{c}}}(L)$ and noticing that $\mathcal{L}%
_{\xi ^{\mathbf{c}}}(a)=-\mu \alpha ,$ $\mathcal{L}_{\xi ^{\mathbf{c}}}L=\mu
L,$ we get: 
\begin{equation*}
(\mathcal{L}_{\xi ^{\mathbf{c}}}\tilde{L})=-\dfrac{1}{\alpha ^{2}}\mu \alpha
L+\dfrac{1}{\alpha }\mu L(y)=0,
\end{equation*}%
i.e., $\xi $ is a Killing vector field for $\tilde{L}$.
\end{proof}

\textbf{Remark. }A conformal vector field for a pseudo-Finsler metric $L$ is
called \textit{essential }if it is not a Killing vector field for any
conformally related metric to $L.$ That is: any essential pseudo-Finslerian
conformal vector field must be lightlike at least at a point.

\bigskip

Passing to Killing vector fields, let us mention the following results due
to Sanchez, \cite{Sanchez}, in Riemannian geometry:

\begin{proposition}
\label{Sanchez1}, \cite{Sanchez}: Let $(M,g)$ be a Lorentzian manifold with
a non-spacelike (at any point)\ Killing vector field $\xi $. If $\xi _{p}=0$
for some $p\in M,$ then $\xi $ vanishes identically.
\end{proposition}

\begin{theorem}
\label{Sanchez 2}, \cite{Sanchez}: If $\xi $ is a Killing vector field on a
Lorentzian manifold $(M,g),$ admitting an isolated zero at some point $p\in
M,$ then, the dimension of $M$ is even and $\xi $ becomes timelike,
spacelike and null on each neighborhood of $p.$
\end{theorem}

Now, using Lemma \ref{conformal_vector_fields_lemma}, the extensions to
pseudo-Finsler spaces of the above results become simple corollaries:

\begin{proposition}
Let $(M,A,L)$ be a Finsler spacetime, with a non-spacelike (at any point)
Killing vector field $\xi $. If $\xi =0$ at some point $p\in M,$ then $\xi $
vanishes identically.
\end{proposition}

\begin{proof}
Since $\xi $ is a Killing vector field for $L,$ it follows from Lemma \ref%
{conformal_vector_fields_lemma} that $\xi $ is a Killing vector field for
the pseudo-Riemannian metric $g^{\xi }.$ But, since the signature of $g^{\xi
}$ coincides with the one of $L,$ $g^{\xi }$ is Lorentzian. The statement
now follows from Proposition \ref{Sanchez1}.
\end{proof}

\begin{theorem}
If $\xi $ is a Killing vector field for a Finsler spacetime $(M,A,L),$
admitting an isolated zero at some point $p\in M,$ then, the dimension of $M$
is even and $\xi $ becomes timelike, spacelike and null on each neighborhood
of $p.$
\end{theorem}

\begin{proof}
Assume $\xi $ is a Killing vector field for $(M,A,L),$ with an isolated zero
at some $p\in M.$ Then, $\xi $ is also a Killing vector for the Lorentzian
metric $g^{\xi }$ on $M$ and 
\begin{equation*}
L(\xi )=g^{\xi }(\xi ,\xi ),
\end{equation*}%
which means that $\xi $ is timelike (respectively, null, spacelike)\ for $L$
iff it is timelike (resp., null, spacelike)\ for $g^{\xi }.$ The result now
follows from Theorem \ref{Sanchez 2}.
\end{proof}

\bigskip

\textbf{Acknowledgment. }The work was supported by a local grant of the 
\textit{Transilvania }University of Brasov.

\end{document}